\documentclass[a4paper]{amsart}
\usepackage[
left=28truemm,
right=28truemm
]{geometry}
\usepackage{amsfonts, amssymb, amsmath, verbatim, amsthm, mathrsfs, amscd, enumerate, ascmac, fancyhdr, caption, hyperref, framed, subfigure}
\usepackage{bbm}
\numberwithin{equation}{section}
\usepackage{tikz}
\usetikzlibrary{positioning,intersections, calc, arrows,decorations.markings, angles}
\theoremstyle{plain}
\newtheorem{theorem}{Theorem}[section]
\newtheorem{lemma}[theorem]{Lemma}
\newtheorem*{lemma*}{Lemma}

\newtheorem*{proposition*}{Proposition}

\newtheorem{conj}[theorem]{Conjecture}
\theoremstyle{definition}

\newtheorem*{claim*}{Claim}

\theoremstyle{remark}

\newtheorem*{Acknowledgements}{Acknowledgements}

\def\d{\mathrm{d}}

\def\C{\mathcal{C}}
\def\K{\mathcal{K}}
\def\1{\mathbbm 1}

\def\C{\mathcal{C}}
\def\K{\mathcal{K}}
\def\1{\mathbbm 1}

\newcommand{\R}{{\mathbb R}}
\newcommand{\Z}{{\mathbb Z}}

\newcommand{\N}{{\mathbb N}}

\author[Bez]{Neal Bez}
\address[Neal Bez]{Department of Mathematics, Graduate School of Science and Engineering,
Saitama University, Saitama 338-8570, Japan}
\email{nealbez@mail.saitama-u.ac.jp}

\author[Kinoshita]{Shinya Kinoshita}
\address[Shinya Kinoshita]{Department of Mathematics, Tokyo Institute of Technology, Meguro-ku, Tokyo, 152-8551, Japan}
\email{kinoshita@math.titech.ac.jp}

\author[Shiraki]{Shobu Shiraki}
\address[Shobu Shiraki]{Departamento de Matem\'atica, Instituto Superior T\'ecnico, Av. Rovisco Pais, Lisboa, 1049-001, Portugal}
\email{shobushiraki@tecnico.ulisboa.pt}
\thanks{
This work was supported by JSPS Kakenhi grant numbers 19H00644, 19H01796, 22H00098 and 23H01080 (Bez), 
21J00514, 22KJ0446 (Kinoshita), and 19H01796, 22H00098 (Shiraki). The third author is also supported by Centro de Análise Matemática, Geometria e Sistemas Dinâmicos (CAMGSD)
}

\begin{document}
\date{\today}
\title[
]
{
A note on Strichartz estimates for the wave equation with orthonormal initial data
}

\subjclass[2010]{}
\begin{abstract}
This note is concerned with Strichartz estimates for the wave equation and orthonormal families of initial data. We provide a survey of the known results and present what seems to be a reasonable conjecture regarding the cases which have been left open. We also provide some new results in the maximal-in-space boundary cases.
\end{abstract}

\maketitle

\section{Introduction}
The (one-sided) wave propagator $e^{it\sqrt{-\Delta}}$ is given by
\[
e^{it\sqrt{-\Delta}}f(x) = \frac{1}{(2\pi)^d} \int_{\mathbb{R}^d} e^{i(x \cdot \xi + t|\xi|)} \widehat{f}(\xi) \, \mathrm{d}\xi
\]
for sufficiently nice initial data $f$ and we often write $Uf(t,x) = e^{it\sqrt{-\Delta}}f(x)$. Our interest here are Strichartz estimates of the form
\begin{equation} \label{e:ONSwave}
\bigg\| \sum_j \lambda_j |Uf_j|^2 \bigg\|_{L^{\frac{q}{2}}_tL^\frac{r}{2}_x} \lesssim \| \lambda \|_{\ell^\beta}
\end{equation}
where $(f_j)_j$ is a (possibly infinite) family of orthonormal functions in the homogeneous Sobolev space $\dot{H}^{s}(\mathbb{R}^d)$, $s = \frac{d}{2} - \frac{d}{r} - \frac{1}{q}$ (we refer the reader forward to the end of this section for the meaning of $\lesssim$). Such an estimate obviously implies the classical Strichartz estimate
\begin{equation} \label{e:classicalSwave}
\| Uf \|_{L^q_tL^r_x} \lesssim \|f\|_{\dot{H}^s}
\end{equation}
by considering a family of data for which all but one member is zero, and rescaling. Therefore, in what follows, we restrict attention to exponents $q$ and $r$ which admit a classical Strichartz estimate. Conversely, it is also clear that \eqref{e:classicalSwave} implies \eqref{e:ONSwave} with $\beta = 1$ by simply using the triangle inequality (and, in particular, not making use of the orthogonality) and so the challenge is to establish \eqref{e:ONSwave} for a range of $\beta$ (depending on $q$ and $r$) which is as large as possible. 

Before going forward to describe the current situation regarding \eqref{e:ONSwave}, we pause to briefly clarify which exponents $q$ and $r$ admit a classical Strichartz estimate \eqref{e:classicalSwave}. For $d \geq 2$, let us say that $(q,r) \in [2,\infty] \times [2,\infty]$ is \emph{admissible} if 
\begin{equation*} 
\frac{1}{q} \leq \frac{d-1}{2}\bigg(\frac{1}{2} - \frac{1}{r}\bigg).
\end{equation*}
Then it is widely known that if $(q,r)$ is admissible and $r < \infty$ then \eqref{e:classicalSwave} holds (see, for example, \cite{KeelTao} and the references therein). The boundary case  
\begin{equation} \label{e:classicalSwaveboundary}
\| Uf \|_{L^q_tL^\infty_x} \lesssim \|f\|_{\dot{H}^{\frac{d}{2} - \frac{1}{q}}}
\end{equation}
presents difficulties. Indeed, it is known that \eqref{e:classicalSwaveboundary} fails when $(q,d) = (4,2)$ and when $q = 2$ for all $d \geq 3$ (see \cite[Theorem 3]{FangWang} for the former claim and \cite[Theorem 1.1]{GLNY} for the latter). Also \eqref{e:classicalSwaveboundary} fails when $q = \infty$ thanks to the failure of the associated Sobolev embedding. For the remaining  $q$ (that is, $q \in (4,\infty)$ when $d = 2$, or $q \in (2,\infty)$ when $d \geq 3$) estimate \eqref{e:classicalSwaveboundary} holds; see, for example \cite[Theorem 3]{FangWang}.

Returning to the wider framework in \eqref{e:ONSwave},
Frank and Sabin \cite{FS_AJM} provided the first progress by establishing this estimate in the sharp admissible case
\begin{equation} \label{e:sharpadmissible}
\frac{1}{q} = \frac{d-1}{2}\bigg(\frac{1}{2} - \frac{1}{r}\bigg)
\end{equation}
with $q = r = \frac{2(d+1)}{d-1}$ and $\beta \leq \frac{2r}{r +2}$. Formally interpolating with the trivial estimate with $\beta = 1$ at $(q,r) = (\infty,2)$, this yields \eqref{e:ONSwave} in the sharp admissible case \eqref{e:sharpadmissible} for $r \in [2, \frac{2(d+1)}{d-1}]$ and $\beta \leq \frac{2r}{r +2}$. As far as we are aware, this interpolation argument seems to require the use of an analytic family of operators (since the regularity exponent $s$ varies); we refer the reader forward to Section \ref{section:strongtype} for arguments of a similar nature.

To discuss results which have been established since \cite{FS_AJM}, it is very convenient to introduce the notation $\beta_\sigma(q,r)$, where $\sigma > 0$, for the exponent satisfying 
\[
\frac{\sigma}{\beta_\sigma(q,r)} = \frac{1}{q} + \frac{2\sigma}{r}.
\]
Observe that $\beta_\sigma(q,r) = \frac{2r}{r + 2}$ in the case $\frac{1}{q} = \sigma(\frac{1}{2} - \frac{1}{r})$, and so the result in \cite{FS_AJM} may be viewed as the statement that \eqref{e:ONSwave} holds for $\beta \leq \beta_{\frac{d-1}{2}}(q,r)$ whenever \eqref{e:sharpadmissible} and  $r \in [2,\frac{2(d+1)}{d-1}]$. It is also trivial to see that $\beta_\sigma(q,r)$ is increasing as a function of $\sigma$ for each fixed $(q,r)$.

In terms of upper bounds on the optimal value of the exponent $\beta$, it was established in \cite{BLN_Forum} that 
\begin{equation} \label{e:necessary}
\beta \leq \min\bigg\{\beta_{\frac{d}{2}}(q,r), \frac{q}{2} \bigg\}
\end{equation}
is necessary in order for \eqref{e:ONSwave} to be true. For completeness, in the Appendix of this note, we give a sketch of the proof. One interesting consequence is that there is no hope of raising $\beta$ above 1 in the case where $q = 2$ and so we omit this case from the subsequent discussion.
We also note that $\beta_{\frac{d}{2}}(q,r) \leq \frac{q}{2}$ if and only if $\frac{d}{r} \leq \frac{d-1}{q}$.

The result in \cite{FS_AJM} was extended in \cite{BLN_Forum} and as a result the following is now known to be true.
\begin{theorem}[see \cite{BLN_Forum, FS_AJM}]  \label{t:mainknown}
Let $d \geq 2$, $q \in (2,\infty), r \in [2,\infty)$, and $\frac{1}{q} \leq \frac{d-1}{2}(\frac{1}{2} - \frac{1}{r})$. Then the estimate \eqref{e:ONSwave} holds in each of the following cases.

\emph{(i)} $\beta \leq \beta_{\frac{d-1}{2}}(q,r)$ when $\frac{d-1}{r} > \frac{d-2}{q}$.

\emph{(ii)} $\beta < \frac{q}{2}$ when $\frac{d-1}{r} \leq \frac{d-2}{q}$.\end{theorem} 
It is perhaps beneficial to view this result with Figures 1 and 2 in mind. The points $A, B, C, D, E$ are given as follows.
\begin{align*}
A = \bigg(\frac{d-2}{2d}, \frac{d-1}{2d}\bigg), \quad 
B  = \bigg(\frac{d-1}{2(d+1)},\frac{d-1}{2(d+1)}\bigg), \\
C = \bigg(\frac{1}{2},0\bigg), \quad D = \bigg(0,\frac{1}{2}\bigg), \quad E = \bigg(\frac{d-3}{2(d-1)},\frac{1}{2}\bigg).
\end{align*}
When $(\frac{1}{r},\frac{1}{q})$ belongs to the interior of the region $OAC$ or the line segment $(A,C]$ (i.e. not including the endpoint at $A$), then Theorem \ref{t:mainknown} guarantees that \eqref{e:ONSwave} holds for $\beta \leq \beta_{\frac{d-1}{2}}(q,r)$. For $d \geq 3$ and in the interior of region $ODEA$, or the line segments $(O,A]$, $[A,E)$, we have that \eqref{e:ONSwave} holds for $\beta < \frac{q}{2}$. Thanks to \eqref{e:necessary}, this means that Theorem \ref{t:mainknown}(ii) is sharp up to the critical case $\beta = \frac{q}{2}$. 

\subsection{A conjecture} 
The most pressing issue appears to be closing the gap between the necessary condition \eqref{e:necessary} and the sufficient condition $\beta \leq \beta_{\frac{d-1}{2}}(q,r)$ given in Theorem \ref{t:mainknown}(i). For this it seems reasonable to us to believe that Theorem \ref{t:mainknown}(i) can be improved up to the necessary condition in \eqref{e:necessary}. This amounts to the following conjecture.
\begin{conj} \label{conj:main}
Let $d \geq 2$, $q \in (2,\infty), r \in [2,\infty)$, and $\frac{1}{q} \leq \frac{d-1}{2}(\frac{1}{2} - \frac{1}{r})$. Then the estimate \eqref{e:ONSwave} holds in each of the following cases.

\emph{(i)} $\beta \leq \beta_{\frac{d}{2}}(q,r)$ when $\frac{d}{r} > \frac{d-1}{q}$.

\emph{(ii)} $\beta < \frac{q}{2}$ when $\frac{d}{r} \leq \frac{d-1}{q}$.
\end{conj}
In Figures 1 and 2, the following points also arise:
\[
A' = \bigg(\frac{d-1}{2(d+1)},\frac{d}{2(d+1)}\bigg), \quad E' = \bigg(\frac{d-2}{2d},\frac{1}{2}\bigg), \quad F = \bigg(\frac{(d-1)^2}{2(d^2+1)},\frac{d(d-1)}{2(d^2+1)}\bigg).
\]
Conjecture \ref{conj:main}(i) says that \eqref{e:ONSwave} holds for $\beta \leq \beta_{\frac{d}{2}}(q,r)$ in the interior of $OFC$ and the line segment $[C,F)$, and Conjecture \ref{conj:main}(ii) says that \eqref{e:ONSwave} holds for $\beta < \frac{q}{2}$ in the interior of $ODEF$, or $(O,F)$, or $[F,E)$.

Observe that the line segment $[C,E']$ corresponds to the sharp admissible line
\[
\frac{1}{q} = \frac{d}{2}\bigg(\frac{1}{2} - \frac{1}{r}\bigg)
\]
for the Strichartz estimates associated with the Schr\"odinger propagator $e^{it\Delta}$. The analogue of \eqref{e:ONSwave} for Schr\"odinger propagator takes the form
\begin{equation} \label{e:ONSSchro}
\bigg\| \sum_j \lambda_j |e^{it\Delta}f_j|^2 \bigg\|_{L^{\frac{q}{2}}_tL^\frac{r}{2}_x} \lesssim \| \lambda \|_{\ell^\beta}
\end{equation}
where $(f_j)_j$ is a family of orthonormal functions in the homogeneous Sobolev space $\dot{H}^{s}(\mathbb{R}^d)$, $s = \frac{d}{2} - \frac{d}{r} - \frac{2}{q}$, and interestingly this family of estimates is much better understood than \eqref{e:ONSwave}. In particular, it is known (see \cite{FLLS, FS_AJM, FS_Survey, BHLNS}) that for $q \in (2,\infty), r \in [2,\infty)$, and $\frac{1}{q} \leq \frac{d}{2}(\frac{1}{2} - \frac{1}{r})$, the estimate \eqref{e:ONSSchro} holds in each of the following cases.

\emph{(i)} $\beta \leq \beta_{\frac{d}{2}}(q,r)$ when $\frac{d}{r} > \frac{d-1}{q}$ (i.e. the interior of $OA'C$ and $[C,A')$).

\emph{(ii)} $\beta < \frac{q}{2}$ when $\frac{d}{r} \leq \frac{d-1}{q}$ (i.e. the interior of $ODE'A'$ and the line segments $(O,A']$, $[A',E')$).

Furthermore, the restriction in \eqref{e:necessary} is also known to be necessary for \eqref{e:ONSSchro} too. Thus, apart from certain critical/boundary cases, we have an almost complete understanding of when the estimates \eqref{e:ONSSchro} are valid. We refer the reader to Frank \emph{et. al} \cite{FLLS} for the origin of this line of study for the Schr\"odinger equation, and to \cite{FS_AJM, FS_Survey, BHLNS} for further developments on \eqref{e:ONSSchro}.


\begin{figure} 

\begin{center}
\begin{tikzpicture}[scale=4]
\draw [<->] (0,2.3) node (yaxis) [left] {$\tfrac{1}{q}$}
|- (2.3,0) node (xaxis) [below] {$\tfrac{1}{r}$};
\draw (0, 0) rectangle (2, 2);
\draw (0,0)--(6/5,8/5)--(2,0)--(0,0);

\node [above] at (1,2) {$E$};
\node [above] at (6/5,2) {$E'$};
\node [above] at (0.075,2) {$D$}; 
\node [above] at (6/5+0.02,8/5) {$A$};
\node [above] at (4/3+0.02,5/3) {$A'$};
\node [right] at (2, 0.06) {$C$};
\node [right] at (4/3,4/3) {$B$};
\node [right] at (16/13,20/13) {$F$};

\node [below] at (1,0) {$\frac{d-3}{2(d-1)}$};
\node [below] at (6/5,0) {$\frac{d-2}{2d}$};

\draw [dashed] (0,1)--(2,2);
\draw [dashed] (0,0)--(2,2);
\draw [dashed] (0,0)--(4/3,5/3);
\draw [dotted] (1,2)--(1,0);
\draw [dotted] (6/5,2)--(6/5,0);

\node [left]at (0,2) {$\frac{1}{2}$};
\node [left]at (0,1) {$\frac14$};
\node[below] at (0,0) {$O$};
\node[below] at (2,0) {$\frac{1}{2}$};
\draw (6/5,8/5)--(1,2);
\draw (6/5,2)--(2,0);

\filldraw[fill=black] (6/5,8/5) circle[radius=0.15mm];
\filldraw[fill=black] (4/3,5/3) circle[radius=0.15mm];
\filldraw[fill=black] (16/13,20/13) circle[radius=0.15mm];
\filldraw[fill=black] (6/5,2) circle[radius=0.15mm];
\filldraw[fill=black] (4/3,4/3)  circle[radius=0.15mm];
\filldraw[fill=black] (0,2)  circle[radius=0.15mm];
\filldraw[fill=black] (2,0)  circle[radius=0.15mm];
\filldraw[fill=black] (1,2)  circle[radius=0.15mm];
\end{tikzpicture}
\end{center}
\caption{The case $d \geq 3$ (note that $D = E$ when $d = 3$). On the lines $[C,E]$ and $[O,A]$ we have $\frac{1}{q} = \frac{d-1}{2}(\frac{1}{2} - \frac{1}{r})$ and $\frac{d-2}{q} = \frac{d-1}{r}$, respectively. }
\label{fig:highDknown}

\begin{center}
\begin{tikzpicture}[scale=3]
\draw [<->] (0,2.3) node (yaxis) [left] {$\tfrac{1}{q}$}
|- (2.3,0) node (xaxis) [below] {$\tfrac{1}{r}$};
\draw (0, 0) rectangle (2, 2);

\draw (2,0)--(0,2);

\node [right]at (0.06,1.01) {$A$};
\node [above]at (2/3 + 0.01,4/3 + 0.02) {$A'$};
\node [right] at (2, 0.06) {$C$};
\node [right] at (0.4, 0.8 + 0.04) {$F$};
\node [right] at (0, 2 + 0.07) {$E'$};
\node [right] at (2/3 + 0.06,2/3+0.02) {$B$};

\draw [dashed] (0,1)--(2,2);
\draw [dashed] (0,0)--(2,2);
\draw [dashed] (0,0)--(2/3,4/3);

%
\node [left]at (0,2) {$\frac{1}{2}$};
\node [left]at (0,1) {$\frac14$};
\node[below] at (0,0) {$O$};
\node[below] at (2,0) {$\frac{1}{2}$};
\draw (2,0)--(0,1);

\filldraw[fill=black] (0.4,0.8) circle[radius=0.15mm];
\filldraw[fill=black] (0,2) circle[radius=0.15mm];
\filldraw[fill=black] (0,1) circle[radius=0.15mm];
\filldraw[fill=black] (2/3,2/3)  circle[radius=0.15mm];
\filldraw[fill=black] (2/3,4/3)  circle[radius=0.15mm];
\filldraw[fill=black] (2,0)  circle[radius=0.15mm];
\end{tikzpicture}
\end{center}
\caption{The case $d = 2$. On the line $[C,A]$ we have $\frac{1}{q} = \frac{1}{2}(\frac{1}{2} - \frac{1}{r})$. }
\label{fig:2Dknown}

\end{figure}

\begin{figure}

\begin{center}
\begin{tikzpicture}[scale=3.5]
\draw [<->] (0,2.3) node (yaxis) [left] {$\tfrac{1}{\beta}$}
|- (2.3,0) node (xaxis) [below] {$\tfrac{1}{r}$};
\draw (0, 0) rectangle (2, 2);



%
\node [left]at (0,2) {$1$};
\node[below] at (0,0) {$0$};
\node[below] at (2,0) {$\frac{1}{2}$};

\draw [fill=gray!70] (2/3,0)--(2/3,2)--(18/17,24/17)--(2,2)--(2,0)--(0,0);
\draw [fill=gray!20] (2/3,2)--(1,3/2)--(2,2);

\draw [dashed] (2,0)--(2/3,2);
\draw [dashed] (2/3,4/3)--(2,2);
\draw [dotted] (1,0)--(1,3/2);
\draw [dotted] (0,3/2)--(1,3/2);
\node [below]at (0.95,0) {$\tfrac{d-2}{2d}$};
\node [left]at (0,3/2+0.03) {$\frac{d-1}{d}$};

\draw [dotted] (18/17,0)--(18/17,24/17);
\draw [dotted] (0,24/17)--(18/17,24/17);
\node [below]at (1.08,0) {$\frac{1}{r_d}$};
\node [left]at (0,24/17-0.03) {$\frac{d(d-1)}{d^2+1}$};

\node[below] at (2/3,0) {$\frac{d-3}{2(d-1)}$};

\node at (1.1,3/2-0.008) {$?$};

\filldraw[fill=black] (2/3,2) circle[radius=0.15mm];
\filldraw[fill=black] (2,0)  circle[radius=0.15mm];
\draw (2/3,2)--(2,2);
\end{tikzpicture}
\end{center}
\caption{For $d \geq 3$, the gap between the necessary condition \eqref{e:necessary} (represented by the dark gray region) and Theorem \ref{t:mainknown}(i) (represented by the light gray region) in the sharp admissible case $\frac{1}{q} = \frac{d-1}{2}(\frac{1}{2} - \frac{1}{r})$. The exponent $r_d$ is given by $r_d = \frac{2(d^2+1)}{(d-1)^2}$.}
\label{fig:highDconj}

\begin{center}
\begin{tikzpicture}[scale=3]
\draw [<->] (0,2.3) node (yaxis) [left] {$\tfrac{1}{\beta}$}
|- (2.3,0) node (xaxis) [below] {$\tfrac{1}{r}$};
\draw (0, 0) rectangle (2, 2);


\draw [thick] (0,1)--(2,2);
\draw (0,0.5)--(2,2);
%
\node [left]at (0,2) {$1$};
\node [left]at (0,1) {$\frac12$};
\node[below] at (0,0) {$0$};
\node[below] at (2,0) {$\frac{1}{2}$};

\node at (0.4, 1) {$?$};

\draw [fill=gray!70] (0,0)--(0,1)--(0.4,0.8)--(2,2)--(2,0)--(0,0);
\draw [fill=gray!20] (0,1)--(0,2)--(2,2)--(0,1);

\draw [dashed] (2,0)--(0,1);
\draw [dashed] (0,0.5)--(2,2);
\draw [dotted] (0.4,0)--(0.4,0.8);
\draw [dotted] (0,0.8)--(0.4,0.8);
\node [left]at (0,0.8) {$\frac25$};
\node[below] at (0.4,0) {$\frac{1}{10}$};

\filldraw[fill=black] (0,1) circle[radius=0.15mm];
\filldraw[fill=black] (2,0)  circle[radius=0.15mm];
\end{tikzpicture}
\end{center}
\caption{For $d = 2$, the gap between the necessary condition \eqref{e:necessary} (represented by the dark gray region) and Theorem \ref{t:mainknown}(i) (represented by the light gray region) in the sharp admissible case $\frac{1}{q} = \frac{1}{2}(\frac{1}{2} - \frac{1}{r})$.}
\label{fig:2Dconj}
\end{figure}

\subsection{Boundary cases}\label{subsection:boundaryintro}

The boundary cases $q = \infty$ and $r = \infty$ have been excluded from the statements of Theorem \ref{t:mainknown} and Conjecture \ref{conj:main}, and for the remainder of this note we mainly focus on the case $r = \infty$. When $q = \infty$, we may simply use the Sobolev inequality for orthonormal functions proved by Lieb \cite{Lieb_Sobolev} to quickly obtain \eqref{e:ONSwave} with $q = \infty$, $r \in (2,\infty)$ and $\beta < \frac{r}{2} = \beta_{\frac{d}{2}}(\infty,r)$; in light of the necessary condition \eqref{e:necessary}, we see that this result is sharp up to the critical case $\beta = \frac{r}{2}$.

When $r = \infty$, matters appear to be less simple and for the remainder of this note we shall be concerned with the estimate
\begin{equation}
\label{est:main_rinfty}
\bigg\| \sum_j \lambda_j |Uf_j|^2 \bigg\|_{L^{\frac{q}{2}}_t L^{\infty}_x} \lesssim \| \lambda \|_{\ell^\beta}, 
\end{equation}
where $(f_j)_j$ is a family of orthonormal functions in $\dot{H}^{\frac{d}{2}-\frac{1}{q}}(\mathbb{R}^d)$. Also we recall from our discussion of the classical Strichartz estimates \eqref{e:classicalSwave} that we are interested in the range $q \in (4,\infty)$ when $d = 2$, and $q \in (2,\infty)$ when $d \geq 3$. As such, we introduce the notation
\[
q_d := \begin{cases}
4 \ & (d=2),\\
2 & (d \geq 3).
\end{cases}
\]
Our first new result is the following.
\begin{theorem} \label{thm:main_strong}
For any $d \geq 2$, the estimate \eqref{est:main_rinfty} holds with $q \in (q_d,\infty)$ and $\beta < \frac{q}{2}$.
\end{theorem}
Although it is unclear to use whether one can extend Theorem \ref{thm:main_strong} to the critical case $\beta = \frac{q}{2}$, we are at least able to show the following restricted weak-type estimates.
\begin{theorem}
\label{thm:WRest}
For any $d\geq 3$ and $q \in (q_d,\infty)$, the estimate
\begin{equation}\label{i:weak-restricted}
\Bigl\|\sum_j\lambda_j|U f_j|^2\Bigr\|_{L_t^{\frac q2,\infty}L_x^\infty}
\lesssim 
\|\lambda\|_{\ell^{\frac q2,1}}
\end{equation}
holds whenever $(f_j)_j$ is a family of orthonormal functions in $\dot{H}^{\frac{d}{2}-\frac{1}{q}}(\mathbb{R}^d)$.
\end{theorem}

The rest of the paper is devoted to proving Theorems \ref{thm:main_strong} and \ref{thm:WRest}, and this is taken up in Sections \ref{section:strongtype} and \ref{section:RWT}, respectively. Prior to that, we introduce a small amount of notation.

\subsection*{Notation}

For non-negative quantities $A$, $B$, the notation $A\lesssim B$ (or $B \gtrsim A$) denotes an estimate $A \leq CB$ for some constant $C > 0$, and we write $A \sim B$ when both $A \lesssim B$ and $B\lesssim A$.

Let $\varphi\in C_0^\infty$ be supported in the annulus $\{\xi\in\mathbb R^n:2^{-1}<|\xi|<2\}$ and satisfy
\[
\sum_{k \in \Z} \varphi(2^{-k} \xi) = 1, \qquad \xi \in \R^d \setminus \{0\}.
\]
Define the family of the Littlewood--Paley operators $(P_k)_{k \in \Z}$ by $\widehat{P_k f}(\xi)=\varphi_k(\xi)\widehat{f}(\xi)$, where $\varphi_k(\xi)=\varphi(2^{-k}\xi)$ for each $k\in \mathbb Z$. Also, let $\chi \in C_0^\infty$ be supported in the ball $\{\xi\in\mathbb R^n: |\xi|<1\}$ with $\chi(0) = 1$.

For $1 \leq p< \infty$, $1\leq q \leq \infty$, we write $L^{p,q}(\mathbb R^d)$ for the associated Lorentz function space and $\ell^{p,q}$ for the associated Lorentz sequence space. We refer the reader to \cite{Grafakos, SteinWeiss} for the definition and wider discussion on these space. In particular, it is well known that
\begin{align*}
    L^{p,p}(\R^d) = L^p(\R^d), \quad L^{p,q_1}(\R^d) \subseteq L^{p,q_2}(\R^d) \quad \text{when } q_1 \leq q_2.
\end{align*}

\section{A sketch proof of Theorem \ref{thm:main_strong}} \label{section:strongtype}
In \cite{BKS}, the analogue of \eqref{est:main_rinfty} for the fractional Schr\"odinger propagators has been obtained.
\begin{theorem}[see \cite{BKS}]\label{thm:main_BKS}
Let $a \in (0,\infty) \setminus \{1\}$ and $d \geq 1$. For $q \in (q_{d+1},\infty)$ and $\beta<\frac{q}{2}$, the estimate
\begin{equation}
\label{est:rinfty_Sch}
\bigg\| \sum_j \lambda_j |e^{it (-\Delta)^\frac a2} f_j|^2 \bigg\|_{L^{\frac{q}{2}}_tL^\infty_x} \lesssim \| \lambda \|_{\ell^\beta}
\end{equation}
holds whenever $(f_j)_j$ is a family of orthonormal functions in $\dot{H}^{\frac{d}{2} - \frac{a}{q}}(\mathbb{R}^d)$. 
\end{theorem}
One may follow the argument in \cite{BKS} to prove Theorem \ref{thm:main_strong} with very minor modifications, and so we only present a very brief sketch here and refer the reader to \cite{BKS} for further details.

Key to the argument is the dispersive estimate of the integral kernel associated with the wave propagator:
\begin{equation}
\label{e:dispersive_W}
    \Bigl|
\int e^{i(x \cdot\xi+t|\xi|)}|\xi|^{-d+\frac{2}{q}+ i \kappa} \chi(\xi) \,\d\xi
\Bigr|
\lesssim
C(\kappa)
|t|^{-\frac2q},
\end{equation}
where $d\geq2$, $\frac{4}{d-1}< q<\infty$, $\kappa\in\mathbb R$, and $C(\kappa)>0$ satisfies $C(\kappa) \lesssim e^{\varepsilon |\kappa|}$ for an arbitrary small $\varepsilon >0$. The corresponding estimate for the fractional Schr\"odinger propagator
\begin{equation}
\label{e:dispersive_S}
    \Bigl|
\int e^{i(x \cdot\xi+t|\xi|^a)}|\xi|^{-d+\frac{2a}{q}+ i \kappa} \chi(\xi) \,\d\xi
\Bigr|
\lesssim
C(\kappa)
|t|^{-\frac2q}
\end{equation}
holds for $d\geq1$,  $a\in(0,\infty)\backslash\{1\}$, $\frac{4}{d}\leq q<\infty$, $\kappa\in\mathbb R$, and $C(\kappa)>0$ as above (for a proof of \eqref{e:dispersive_W} and \eqref{e:dispersive_S}, see for example \cite{GPW} and \cite{GLNY}). Notice that the condition on $q$ in \eqref{e:dispersive_W} corresponds to the one for \eqref{e:dispersive_S}, but $d$ is replaced by $d-1$ and the endpoint $q=\frac{4}{d-1}$ is excluded; the lack of an endpoint estimate here turns out to be harmless since our goal is $\beta<\frac q2$. 

The argument proceeds using a duality principle observed by 
Frank--Sabin \cite[Lemma~3]{FS_AJM} in an abstract setting. In our case, we see that
\eqref{est:main_rinfty} with $U$ replaced by $\chi(|D|)U$ is equivalent to
\begin{equation}
    \label{est_main_Dual}
    \|WU\chi^2(|D|)|D|^{-(d-1+\frac{2}{\widetilde{q}})}U^*\overline{W}\|_{\C^{\beta'}}
\lesssim 
\|W\|_{L_t^{\widetilde{q}}L_x^{2}}^2,
\end{equation}
where $\beta'$ denotes the (standard) conjugate of $\beta$, and $\widetilde{q}$ denotes the (half-)conjugate of $q$ given by
\[
\frac1\beta+\frac{1}{\beta'}=1, \qquad \frac{1}{q}+\frac{1}{\widetilde{q}}=\frac12.
\]
Here, $\C^{\beta'}$ with $\beta' \geq 1$ denotes the Schatten space on $L^2(\mathbb R^{d+1})$ defined as
\[
\C^{\beta'}=\{T \in \mathrm{Com}(L^2(\R^{d+1})) : \|T\|_{\C^{\beta'}} := \|(\mu_j(T))_{j\in \N} \|_{\ell^{\beta'}} < \infty\},
\]
where $\mathrm{Com}(L^2(\R^{d+1}))$ denotes the set of compact operators on $L^2(\R^{d+1})$ and $(\mu_j(T))_j$ denotes the family of the singular values of $T$ (i.e. non-zero eigenvalues of $\sqrt{T^*T}$). Note that establishing \eqref{est:main_rinfty} with $\chi(|D|)U$ suffices since one can rescale to deduce the same estimate with $\chi(\varepsilon|D|)U$ for any $\varepsilon > 0$, and then take a limit $\varepsilon \to 0$.

The space $\C^1$ consists of trace-class operators and $\C^\infty$ consists of compact operators (endowed with the standard operator norm). 
Hence $\C^{\beta'}$ with $1<\beta' < \infty$ is regarded as an intermediate space. Of particular importance to the analysis in this note is that the class $\C^2$ consists of Hilbert--Schmidt operators and that there is a very useful expression of the $\C^2$ norm; if $T$ is given 
\[
T f (x) = \int_{\R^{d+1}} K(x,y)f(y) \, \mathrm{d}y, \qquad f \in L^2(\R^{d+1}),
\]
then $\|T\|_{\C^2} = \|K\|_{L^2(\mathbb R^{d+1}\times\mathbb R^{d+1})}$. 
For more details about Schatten spaces, see Simon's textbook \cite{Simon}.

Our goal is to prove \eqref{est_main_Dual} for $\widetilde{q} \in (2,\widetilde{q}_d)$ and $\beta'>\frac{\widetilde{q}}{2}$. 
We write $\mathcal{U} = \chi(|D|)U$ and divide into two cases depending on $\widetilde{q}$.

\textbf{Case 1: $d \geq 3$ and $4 \leq \widetilde{q} <\infty$:}
We aim to get the following $\C^{2}$ and $\C^\infty$ bilinear estimates: if $2 < \widetilde{q_0} < 4$, $2 < \widetilde{q_1} < \infty$,
\begin{align}\label{i:C2}
& \|W_1\mathcal{U}|D|^{-(d-1+\frac{2}{\widetilde{q_0}})+i\kappa}\mathcal{U}^*\overline{W}_2\|_{\C^2}
\lesssim 
C(\kappa) 
\|W_1\|_{L_t^{\widetilde{q_0},4}L_x^2}\|W_2\|_{L_t^{\widetilde{q_0},4}L_x^2},\\
\label{i:Cinfty}
& \|W_1\mathcal{U}|D|^{-(d-1+\frac{2}{\widetilde{q_1}})+i\kappa}\mathcal{U}^*\overline{W}_2\|_{\C^\infty}
\lesssim 
C(\kappa) 
\|W_1\|_{L_t^{\widetilde{q_1},\infty}L_x^2}\|W_2\|_{L_t^{\widetilde{q_1},\infty}L_x^2},
\end{align}
where $C(\kappa) \lesssim e^{\varepsilon |\kappa|}$ for an arbitrary small $\varepsilon >0$. 
Then, by applying a bilinear version of Stein's analytic interpolation to the above estimates, we may obtain \eqref{est_main_Dual} for the case $4 \leq \widetilde{q} <\infty$. For this purpose, the complex power $i \kappa$ is essential in \eqref{i:C2} and \eqref{i:Cinfty}. 
One may also notice that, thanks to the Lorentz improvements on the right-hand sides of \eqref{i:C2} and \eqref{i:Cinfty}, the outcome is also slightly stronger than \eqref{est_main_Dual}, namely, in Case 1 and for $\beta'>\frac{\widetilde{q}}{2}$ one obtains
\[
\|W\mathcal{U}|D|^{-(d-1+\frac{2}{\widetilde{q}})}\mathcal{U}^*\overline{W}\|_{\C^{\beta'}}
\lesssim 
\|W\|_{L_t^{\widetilde{q}, 2\beta'}L_x^{2}}^2.
\]
The proofs of \eqref{i:C2} and \eqref{i:Cinfty} are based on the dispersive estimate for the wave propagator \eqref{e:dispersive_W} and the Lorentz refinement of the Hardy--Littlewood--Sobolev inequality due to O'Neil \cite{ONeil}. 
The interested reader is encouraged to visit \cite{BKS} where further details regarding the bilinear analytic interpolation of \eqref{i:C2} and \eqref{i:Cinfty} can also be found.

\textbf{Case 2: $d \geq 2$ and $2 < \widetilde{q} < 4$:}
We approach this case by first establishing the somewhat weaker estimates
\begin{equation}
    \label{est:Lowerhalf}
    \|W_1\mathcal{U}|D|^{-(d-1+\frac{2}{\widetilde{q}})+i\kappa}\mathcal{U}^*\overline{W}_2\|_{\C^{\beta'}}
\lesssim 
C(\kappa) 
\|W_1\|_{L_t^{\widetilde{q},2}L_x^2}\|W_2\|_{L_t^{\widetilde{q},2}L_x^2},
\end{equation}
where $\beta'>\frac{\widetilde{q}}{2}$ and $C(\kappa) \lesssim e^{\varepsilon |\kappa|}$ for an arbitrary small $\varepsilon >0$. In fact, we can offset the loss in \eqref{est:Lowerhalf} (in the sense of $L_t^{\widetilde{q},2} \subsetneq L_t^{\widetilde{q}}$ when $\widetilde{q} > 2$) by capitalizing on the gain in \eqref{i:C2} (in the sense of $L_t^{\widetilde{q_0}} \subsetneq L_t^{\widetilde{q_0},4}$ if $\widetilde{q_0}< 4$), and use bilinear analytic interpolation once again to obtain the desired estimate \eqref{est_main_Dual} in Case 2 (see \cite{BKS} for further details of this step).

To prove \eqref{est:Lowerhalf}, we use a frequency-localization argument and use a different bilinear interpolation argument in this spirit of Keel--Tao \cite{KeelTao} based on the estimates
\begin{align}\label{i:C1 loc}
    \|W_1\mathcal{U}P_k|D|^{-2s+i\kappa}P_k\mathcal{U}^*\overline{W}_2\|_{\C^1}
    &\lesssim
    (2^k)^{d-2s}
    \|W_1\|_{L_t^{2}L_x^2}\|W_2\|_{L_t^{2}L_x^2},\\
    \label{i:C2 loc}
    \|W_1\mathcal{U}P_k|D|^{-2s+i\kappa}P_k\mathcal{U}^*\overline{W}_2\|_{\C^2}
    &\lesssim
    C(\kappa)
    (2^k)^{d-2s-1+\frac{1}{\widetilde{q_1}}+\frac{1}{\widetilde{q_2}}}
    \|W_1\|_{L_t^{\widetilde{q_1}}L_x^2}\|W_2\|_{L_t^{\widetilde{q_2}}L_x^2},
\end{align}
where $\widetilde{q_1},\widetilde{q_2}\in[2,\infty)$ are such that $\frac{1}{\widetilde{q_1}}+\frac{1}{\widetilde{q_2}}>\frac12$, and $C(\kappa) \lesssim e^{\varepsilon |\kappa|}$ for an arbitrary small $\varepsilon >0$. 
As in Case 1, \eqref{i:C2 loc} readily follows from a stronger form of the dispersive estimate \eqref{e:dispersive_W} with decay $(1 + |t|)^{-d/2}$. 
The estimate \eqref{i:C1 loc} is a manifestation of the frequency localization\footnote{Note that \eqref{i:C1 loc} without the frequency localization is not permitted due to the failure of the associated Sobolev embedding.} and Bessel's inequality for orthonormal sequences, and for its proof, it is convenient to make use of the representation
\[
\| T\|_{\mathcal{C}^{\beta'}} = \sup_{(\phi, \psi) \in \mathcal{B}} \|\langle  T \phi_j,   \psi_j \rangle_{L^2(\R^{d+1})} \|_{\ell^{\beta'}}
\]
for Schatten norms (see \cite[Proposition~2.6]{Simon}). Here, $\mathcal{B}$ denotes the set of pairs of orthonormal sequences in $L^2(\R^{d+1})$. 

\section{Proof of Theorem \ref{thm:WRest}} \label{section:RWT}
Key to our proof of \eqref{i:weak-restricted} is a clever summation idea going back to Bourgain \cite{Bourgain} (which has been used in several papers; see also \cite[Section 6.2]{CSWW} and \cite[Lemma 2.3]{Lee03} for example) which allows us to pass from frequency-local to frequency-global estimates.
\begin{lemma}[see \cite{BHLNS}]\label{l:Bourgain's trick}
Let $q_0, q_1,r \in [2,\infty]$, $\beta_0, \beta_1 \in [2,\infty]$ and $(g_j)_j$ be a uniformly bounded sequence in $L_t^{q_0}L_x^{r}\cap L_t^{q_1}L_x^{r}$. Suppose there exist $\varepsilon_0$, $\varepsilon_1>0$ such that 
\[
\Big\|\sum_j\lambda_j|P_kg_j|^2\Big\|_{L_t^{\frac{q_0}{2},\infty}L_x^{\frac{r}{2}}}\lesssim 2^{-\varepsilon_0k}\|\lambda\|_{\ell^{\beta_0}}
\]
and
\[
\Big\|\sum_j\lambda_j|P_kg_j|^2\Big\|_{L_t^{\frac{q_1}{2}, \infty}L_x^{\frac{r}{2}}}\lesssim 2^{\varepsilon_1k}\|\lambda\|_{\ell^{\beta_1}}
\]
for all $k\in\mathbb{Z}$, then
\[
\Big\|\sum_j\lambda_j|g_j|^2\Big\|_{L_t^{\frac{q}{2},\infty}L_x^{\frac{r}{2}}}\lesssim \|\lambda\|_{\ell^{\beta,1}},
\]
where $\frac{1}{q}=\frac{1-\theta}{q_0}+\frac{\theta}{q_1}$, $\frac{1}{\beta}=\frac{1-\theta}{\beta_0}+\frac{\theta}{\beta_1}$ and $\theta=\frac{\varepsilon_0}{\varepsilon_0+\varepsilon_1}$.
\end{lemma}
The above lemma was proved in \cite[Proposition 2.2]{BHLNS}. Let us give a sketch of the argument when $q_1 = \infty$ (since this case does not immediately follow from the argument in \cite{BHLNS} as it stands).
 
\begin{proof}[Proof of Lemma \ref{l:Bourgain's trick} when $q_1 = \infty$]
It is enough to show the desired inequality for $\lambda = \textbf{1}_E$ (here, $\textbf{1}_E(j) = 0$ if $j \in E$, and zero otherwise), so that we aim for 
\begin{equation} \label{e:BourgainETS}
|J_\nu|
\lesssim 
\left(
\frac{(\# E)^\frac1\beta}{\nu}
\right)^\frac q2,
\end{equation}
where $\nu>0$ and
\[
J_\nu
:=
\Big\{t\in\mathbb R: \Big\|\sum_{j\in E}|g_j(t,\cdot)|^2\Big\|_{L_x^\frac r2}>\nu \Big\}.
\]
For a fixed $M\in\mathbb Z$ (chosen later), set 
\[
J_{\nu,M}^0
:=
\Big\{t\in\mathbb R: \Big\|\sum_{j\in E}|\sum_{k\geq M+1}P_k g_j(t,\cdot)|^2\Big\|_{L_x^\frac r2}>\frac\nu2 \Big\}
\]
and
\[
J_{\nu,M}^1
:=
\Big\{t\in\mathbb R: \Big\|\sum_{j\in E}|\sum_{k\leq M} P_k g_j(t,\cdot)|^2\Big\|_{L_x^\frac r2}>\frac\nu2 \Big\}.
\]
Regarding $J_{\nu,M}^0$, let us observe that 
\[
 \Big\|\sum_{j\in E}|\sum_{k \geq M+1}P_k g_j(t,\cdot)|^2\Big\|_{L_t^{\frac{q_0}{2}, \infty} L_x^\frac r2}
 \lesssim
 2^{-\varepsilon_0M}(\# E)^\frac{1}{\beta_0}.
\]
In fact, applying Minkowski's inequality twice gives
\[
 \Big\|\sum_{j\in E}|\sum_{k\geq M+1} P_k g_j(t,\cdot)|^2\Big\|_{L_t^{\frac{q_0}{2}, \infty} L_x^\frac r2}
 \lesssim
\Big(\sum_{k\geq M+1}\Big\|\sum_{j}|P_kg_j|^2\Big\|_{L_t^{\frac{q_0}{2}, \infty} L_x^{\frac r2}}^\frac12\Big)^2,
\]
which is further bounded by (up to some constant)
$
2^{-\varepsilon_0M} (\# E)^\frac{1}{\beta_0}.
$
Hence, one readily sees that 
\[
|J_{\nu,M}^0| 
\lesssim \nu^{-\frac{q_0}{2}} \Big\|\sum_{j\in E}|\sum_{k\geq M+1} P_k g_j(t,\cdot)|^2\Big\|_{L_t^{\frac{q_0}{2}, \infty} L_x^\frac r2}^{\frac{q_0}{2}}
\lesssim 2^{-\frac{\varepsilon_0Mq_0}{2}}\nu^{-\frac{q_0}{2}}(\# E)^\frac{q_0}{2\beta_0}.
\]
A similar calculation reveals that there exists $C>0$ such that
\[
\Big\|\sum_{j\in E}|\sum_{k\leq M} P_k g_j(t,\cdot)|^2\Big\|_{L_t^{\infty} L_x^\frac r2}
\leq 
C 2^{\varepsilon_1M}(\# E)^\frac{1}{\beta_1},
\]
and now we select $M$ to be the largest integer such that 
\[
C 2^{\varepsilon_1M}(\# E)^\frac{1}{\beta_1}
<
\frac\nu2.
\]
With this choice, we have $J_{\nu,M}^1 = \emptyset$ and
$
2^{M}
\sim
\nu^{\frac{1}{\varepsilon_1}} (\# E)^{-\frac{1}{\beta_1 \varepsilon_1}}.
$
Since $|J_\nu|\leq |J_{\nu,M}^0|$, we obtain \eqref{e:BourgainETS} after a straightforward computation.
\end{proof}

\begin{proof}[Proof of Theorem~\ref{thm:WRest}]
Let $(f_j)_{j \in \N}$ be a family of orthonormal functions in $L^2(\R^d)$. First note that
\begin{equation}\label{i:loc bottom}
\Big\|\sum_j\lambda_j|U |D|^{-s} P_kf_j|^2\Bigr\|_{L_t^\infty L_x^\infty}
\lesssim
(2^k)^{d-2s}
\|\lambda\|_{\ell^\infty}.
\end{equation}
This is the estimate dual to \eqref{i:C1 loc} (which can easily be verified directly by an application of Bessel's inequality as in \cite{BHLNS}).

First we consider the case $d\geq 4$ in which case it is known (see for example \cite[Theorem 1]{FangWang}) that the classical Strichartz estimate \eqref{e:classicalSwaveboundary} holds with $f$ replaced by $P_0f$. Thus, by recaling, 
\begin{equation}
\label{i:loc_top_single}
    \| U|D|^{-s}P_kf_j \|_{L_t^2L_x^\infty}
\lesssim
(2^k)^{\frac{d-1}{2}-s}
\end{equation}
holds for each $f_j$, and therefore the triangle inequality implies
\begin{equation}\label{i:loc top}
\Bigl\|\sum_j \lambda_j|U|D|^{-s}P_kf_j|^2\Bigr\|_{L_t^1L_x^\infty}
\lesssim
(2^k)^{d-1-2s}
\|\lambda\|_{\ell^1}.
\end{equation}
Although \eqref{i:loc bottom} and \eqref{i:loc_top_single} are valid for all $s \in \mathbb{R}$, if we fix $q\in(2,\infty)$ and apply these estimates with $s_q=\frac d2-\frac1q$, then Lemma \ref{l:Bourgain's trick} immediately gives the desired estimate 
\[
\Bigl\|\sum_j\lambda_j|U|D|^{-s_q}f_j|^2\Bigr\|_{L_t^{\frac q2,\infty}L_x^\infty}
\lesssim 
\|\lambda\|_{\ell^{\frac q2,1}}.
\]

The case when $d=3$ requires a bit more work to show \eqref{i:weak-restricted} for $2<q<\infty$ since \eqref{i:loc_top_single} with $d=3$ is known to be false (see \cite{MontSmith}). 
Thus, instead of \eqref{i:loc top}, we aim for
\begin{equation}\label{i:loc top_3d}
\Bigl\|\sum_j \lambda_j|U|D|^{-s}P_kf_j|^2\Bigr\|_{L_t^{\frac{q}{2}} L_x^\infty}
\lesssim
(2^k)^{3-\frac2q -2s}
\|\lambda\|_{\ell^{\frac{q}{2}}}
\end{equation}
with $q$ larger than but abritrarily close to $2$ (here the implicit constant blows up as $q \to 2$). It is clear that Lemma \ref{l:Bourgain's trick} combined with \eqref{i:loc bottom} and \eqref{i:loc top_3d} yields \eqref{i:weak-restricted}. 

Let us check \eqref{i:loc top_3d}. First, by rescaling, we may assume $k=0$. Using duality \cite[Lemma 3]{FS_AJM}, it suffices to show
\begin{equation}\label{est:WRest_Dual_3d2}
\|WS\overline{W}\|_{\C^{2^m}}
\lesssim 
\|W\|_{L_t^{2^{m+1}}L_x^{2}}^2,
\end{equation}
where $S:=U P_0 P_0^* U^*$ and $m$ a sufficiently large natural number. Recall that the $\C^{2}$-norm can be expressed by the $L^2$-norm of its integral kernel, and in addition, from the definition of singular values, it holds that $\| T \|_{\C^{2^{m+1}}}^2 = \| T^* T \|_{\C^{2^m}}$. 
Therefore, since $S$ is self-adjoint, we have
\[
\|WS\overline{W}\|_{\C^{2^m}}^{2^{m-1}} = \| (WS\overline{W})^{2^{m-1}}\|_{\C^{2}}.
\]
Hence our task is to get an $L^2$ bound for the integral kernel $\K_{2^{m-1}}$ of $(WS\overline{W})^{2^{m-1}}$. The notation for the variables in the forthcoming argument requires a little care and our choices will become clear as we proceed.

Let us write $\zeta_j=(t_j,x_j)\in \mathbb R^{1+3}$. Then
\[
(WS\overline{W})F(\zeta_1)
=
\int_{\mathbb R^{4}}F(\zeta_2)\K_1(\zeta_1,\zeta_2)\,\d\zeta_2
\]
with $\K_1(\zeta_1,\zeta_2)=W(\zeta_1)\overline{W}(\zeta_2)K(\zeta_1,\zeta_2)$ and 
\[
K(\zeta_1,\zeta_2)
=
\int_{\mathbb R^3}|\varphi(\xi)|^2 e^{i((x_1-x_2)\cdot\xi+(t_1-t_2)|\xi|)}\,\d\xi.
\]
To handle this, we shall invoke the following frequency-localized dispersive estimate (see e.g. \cite{GPW}):
\begin{equation}
\label{e:dispersive_W_local}
    \sup_{x_1,x_2} |K(\zeta_1,\zeta_2)| \lesssim \langle t_1-t_2 \rangle^{-1},
\end{equation}
where $\langle t \rangle := 1 + |t|$.

For $m = 2$ we have
\begin{align*}
(WS\overline{W})^2 G(\zeta_1)
=
\int_{\mathbb R^{4}}G(\zeta_3)\Big(\int_{\mathbb R^{4}}W(\zeta_1)|W(\zeta_2)|^2\overline{W}(\zeta_3)K(\zeta_1,\zeta_2)K(\zeta_2,\zeta_3)\,\d\zeta_2\Big)\d\zeta_3
\end{align*}
and from this we have an expression for $\K_2(\zeta_1,\zeta_3)$. Inductively, 
\begin{align*}
\K_{2^{m}}(\zeta_1,\zeta_{2^{m}+1})
&=
\int_{\mathbb R^{4}}\K_{2^{m-1}}(\zeta_1,\zeta_{2^{m-1}+1})\K_{2^{m-1}}(\zeta_{2^{m-1}+1},\zeta_{2^{m}+1})\,\d\zeta_{2^{m-1}+1}\\
&=
\int_{(\mathbb R^{4})^{2^m-1}}W(\zeta_1)\prod_{j=2}^{2^m}|W(\zeta_j)|^2\overline{W}(\zeta_{2^m+1})
\prod_{k=1}^{2^m}K(\zeta_k,\zeta_{k+1})\, \prod_{\ell = 2}^{2^m} \d\zeta_\ell
\end{align*}
for all $m \geq 1$, and consequently
\begin{align*}
|\K_{2^{m}}(\zeta_1,\zeta_{2^{m}+1})| \lesssim |W(\zeta_1)| \bigg( \int_{\mathbb R^{2^m-1}} \prod_{j=2}^{2^m} h(t_j)  \prod_{k=1}^{2^m} \langle t_k-t_{k + 1} \rangle^{-1} \, \prod_{\ell = 2}^{2^m} \d t_\ell \bigg) |W(\zeta_{2^m+1})|,
\end{align*}
where $h(t) := \| W(t,\cdot)\|_{L^2_x}^2$.

Now we are ready to estimate the relevant Schatten norm of $WS\overline{W}$:
\begin{align*}
& \|WS\overline{W}\|_{\C^{2^{m}}}^{2^{m}}\\
&=\int_{(\mathbb{R}^4)^2} |\K_{2^{m-1}}(\zeta_1,\zeta_{2^{m-1}+1})|^2 \, \d\zeta_1 \d\zeta_{2^{m-1}+1}   \\
&\lesssim \int_{\mathbb{R}^2} h(t_1) \bigg(\int_{\mathbb{R}^{2^{m-1}-1}} \prod_{j=2}^{2^{m-1}} h(t_j)  \prod_{k=1}^{2^{m-1}} \langle t_k-t_{k + 1} \rangle^{-1} \, \prod_{\ell = 2}^{2^{m-1}} \d t_\ell  \bigg)^2 h(t_{2^{m-1}+1}) \, \d t_1 \d t_{2^{m-1}+1}
\end{align*}
By expanding the square and an appropriate choice of labelling\footnote{If we use $t_2,\ldots,t_{2^{m-1}},t^*_2,\ldots,t^*_{2^{m-1}}$ for the variables arising from expanding the square, then the claimed formula follows by relabelling $t_j = t^*_{2^m + 2-j}$ for $j = 2^{m-1}+2,\ldots,2^m$.}, we see that
\begin{align}\label{i:LHS BL}
\|WS\overline{W}\|_{\C^{2^{m}}}^{2^{m}}
\lesssim
\int_{\mathbb R^{2^{m}}}\prod_{j=1}^{2^{m+1}} f_j(L_j\textbf{t})\,\d \textbf{t},
\end{align}
where $L_j:\mathbb R^{2^{m}}\ni \textbf{t}\mapsto v_j\cdot\textbf{t}\in\mathbb R$ with
$$v_j
=
\begin{cases}
e_j \quad &\text{if $ j=1,\dots,2^{m}$},\\
e_{j-2^{m}}-e_{j+1-2^{m}} \quad &\text{if $j=2^{m}+1,\dots,2^{m+1}$}, 
\end{cases}
$$
under the relation $e_{2^{m}+1}=e_1$, and 
\[
f_j(t)
=
\begin{cases}
h(t)\quad&\text{if $j=1,\dots, 2^{m}$,}\\
\langle t \rangle^{-1}\quad &\text{if $j=2^{m}+1,\dots, 2^{m+1}$}.
\end{cases}
\] 
One may notice that the term on the right-hand side of \eqref{i:LHS BL} has the structure of the left-hand side of the Brascamp--Lieb inequality and, in fact, we shall invoke the following characterization of finiteness of the Brascamp--Lieb constant (in the rank-one case) due to Barthe \cite{Barthe98}. 
Let $n$, $M \in \N$ and for $I=\{1, \ldots, M \}$ we denote by $\textbf{1}_{I}$ the vector of $\mathbb R^{M}$ such that $j$-th element of $\textbf{1}_{I}$ is $1$ if $j\in I$ and $0$ elsewhere.
\begin{lemma}[see Proposition~3 in \cite{Barthe98}]\label{l:Barthe}
There exists $C<\infty$ such that 
\[
\int_{\mathbb R^n}\prod_{j=1}^M f_j(x\cdot v_j)\,\d x
\leq
C\prod_{j=1}^M\|f_j\|_{L^{p_j}(\mathbb R)}
\]
holds for all non-negative $f_j$ in $L^{p_j}(\mathbb{R})$ if and only if the following holds. The vector $\frac1p=(\frac{1}{p_1}\,\dots,\frac{1}{p_M})$ belongs to the convex hull of $\textbf{1}_I$ for subsets such that $(v_i)_{i\in I}$ forms a basis of $\mathbb R^n$.
\end{lemma}
We will use this lemma with $n= 2^m$ and $M=2^{m+1}$. 
Let $I_0=\{2^{m}+1,2^{m}+2, \dots,2^{m+1}\}$ so that $\textbf{1}_{I_0}=(0,\dots,0;1,\dots,1)$; i.e. the first half of elements are all $0$ and the rests are all $1$ (the separator ``;" is positioned so that the components have been divided into exactly two equal-sized parts). 
For integers $i$, $j$ with $1\leq i \leq 2^m <  j\leq 2^{m+1}$ the transposition operator swapping $i$-th element and $j$-th element, denoted by $[i,j]$, provides
\[
[i,j]\textbf{1}_{I_0}
=
(0,\dots,0,1,0,\dots,0;1,\dots,1,0,1,\dots,1).
\]

From the definition of $(v_j)_j$, if $1\leq i \leq 2^m$, the vectors 
\[
(v_j)_{j \in (\{i\} \cup (I_0\setminus \{2^m + i\}))}
\]
form a basis of $\R^{2^m}$. 
In $\mathbb R^{2^{m+1}}$, considering $2^{m}$ points $[1,2^{m}+1]\textbf{1}_I, \dots, [2^{m},2^{m+1}]\textbf{1}_I$, one may discover that the middle point of the plane spanned by those vectors, namely, 
\[
(2^{-{m}},\dots,2^{-{m}};1-2^{-{m}},\dots, 1-2^{-{m}})
\]
lies on the boundary of the convex hull of characteristic vectors. 
Hence, if we set
\begin{align*}
(p_1^{-1},\dots, p_{2^{m}}^{-1};p_{2^{m}+1}^{-1},\dots,p_{2^{m+1}}^{-1}) =
(2^{-{m}},\dots,2^{-{m}};1-2^{-{m}},\dots, 1-2^{-{m}}),
\end{align*}
then Lemma \ref{l:Barthe} implies that
\begin{align*}
\int_{\mathbb R^{2^{m}}}\prod_{j=1}^{2^{m+1}}f_j(L_j\textbf{t})\,\d\textbf{t}
\lesssim
\prod_{j=1}^{2^{m+1}}\|f_j\|_{L_t^{p_j}} \lesssim \|W\|_{L_t^{^{2^{m+1}}}L_x^2}^{2^{m+1}},
\end{align*}
which ends the proof.
\end{proof}

\section*{Appendix: Necessary conditions} \label{appendix}
We provide a justification of the necessary condition \eqref{e:necessary} for \eqref{e:ONSwave}. The proof here is not new and may be found in \cite[Proposition 9]{BLN_Forum}. 

\subsection*{The necessity of $\beta \leq \frac{q}{2}$}
The argument we give here is the same as in the proof of \cite[Proposition 9]{BLN_Forum}; since $r < \infty$ was a blanket assumption throughout \cite{BLN_Forum}, here we consider the only case $r = \infty$.

Consider the family $(f_j)_{j \geq 1}$ given by 
\[
f_j(x) = Ug(-j,x)
\]
where the function $g$ satisfies
\[
\widehat{g}(\xi) = c\frac{\mathrm{1}_{[\pi,3\pi]}(|\xi|)}{|\xi|^{s + \frac{d-1}{2}}}
\]
and the constant $c > 0$ will be chosen shortly. Parseval's identity and changing to polar coordinates reveals 
\[
\langle f_j,f_k \rangle_{\dot{H}^s} = \frac{1}{(2\pi)^d}\int_{\mathbb{R}^d} e^{i(j-k)|\xi|}|\widehat{g}(\xi)|^2 |\xi|^{2s} \, \mathrm{d}\xi =  \frac{c^2|\mathbb{S}^{d-1}|}{(2\pi)^d}\int_\pi^{3\pi} e^{i(k-j)\rho} \, \mathrm{d}\rho
\]
and so $(f_j)_{j \geq 1}$ is an orthonormal family in $\dot{H}^s(\mathbb{R}^d)$ upon an appropriate choice of $c$.

Also, assuming $\lambda_j > 0$ for each $j$ we have
\begin{align*}
\bigg\| \sum_j \lambda_j |Uf_j|^2 \bigg\|_{L^{\frac{q}{2}}_tL^\infty_x}^{\frac{q}{2}} & \geq \int_{\mathbb{R}} \bigg( \sum_j \lambda_j |Ug(t-j,0)|^2 \bigg)^{\frac{q}{2}} \mathrm{d}t \\
& \geq \sum_{n \geq 1} \lambda_n^{\frac{q}{2}} \int_n^{n+\varepsilon_0}   |Ug(t-n,0)|^q \, \mathrm{d}t \\
& = C \sum_{n \geq 1} \lambda_n^{\frac{q}{2}} 
\end{align*}
where $\varepsilon_0 > 0$ is a sufficiently small constant and
\[
C = \int_0^{\varepsilon_0}   |Ug(t,0)|^q \, \mathrm{d}t  = c^q  \int_0^{\varepsilon_0} \bigg| \int_{|\xi| \in [\pi,3\pi]} e^{it|\xi|} \frac{\mathrm{d}\xi}{|\xi|^{s + \frac{d-1}{2}}} \bigg|^q \mathrm{d}t.
\]
The constant $C$ is clearly finite, and a sufficiently small choice of $\varepsilon_0$ guarantees $|t|\xi|| \leq \frac{1}{10}$ for $\xi \in [\pi,3\pi]$ and thus
\begin{align*}
C & \geq c^q  \int_0^{\varepsilon_0} \bigg| \int_{|\xi| \in [\pi,3\pi]} \cos(t|\xi|) \frac{\mathrm{d}\xi}{|\xi|^{s + \frac{d-1}{2}}} \bigg|^q \mathrm{d}t \\
& \geq  \bigg(\frac{c}{2} \bigg)^q \int_0^{\varepsilon_0} \bigg(\int_{|\xi| \in [\pi,3\pi]}  \frac{\mathrm{d}\xi}{|\xi|^{s + \frac{d-1}{2}}} \bigg)^q \mathrm{d}t
> 0.
\end{align*}
Hence, if \eqref{e:ONSwave} holds we deduce $\|\lambda\|_{\frac{q}{2}} \lesssim \|\lambda\|_{\beta}$ and this implies $\beta \leq \frac{q}{2}$.

\subsection*{The necessity of $\beta \leq \beta_{\frac{d}{2}}(q,r)$}
Let $B(x,r)$ be the ball in $\mathbb{R}^d$ centred at $x$ with radius $r$. Consider the family $(f_j)_{j = 1}^N$ given by 
\[
\widehat{f_j}(\xi) = \mathrm{1}_{B(0,1)}(R(\xi - v_j)).
\]
Here, $v_1,\ldots,v_N \in \mathbb{R}^d$ are fixed vectors satisfying $|v_j| \in [1,2]$ and chosen so that the balls $\{ B(v_j,\frac{1}{R}) \}_{j=1}^N$ are disjoint. In particular, this means that the functions $(\widehat{f_j})_{j = 1}^N$ have disjoint support and thus $(f_j)_{j = 1}^N$ is an orthogonal family. Moreover, we choose a maximal collection of vectors $v_1,\ldots,v_N$ under the above constraints, which means $N \sim R^d$. 

Next we claim that $|Uf_j| \gtrsim R^{-d} 1_{B(0,c'R)}$ for an appropriately small choice of the constant $c'$. To see this, first assume that $v_j = \rho e_d$ for some $\rho \in [1,2]$, and write
\begin{align*}
x \cdot \xi + t|\xi| & = \bar{x} \cdot \bar{\xi} + (x_d + t)(\xi_d - \rho) + t(|\xi| - \xi_d) + \rho(x_d + t) \\
& = \bar{x} \cdot \bar{\xi} + (x_d + t)(\xi_d - \rho) + t\frac{|\bar{\xi}|^2}{|\xi| + \xi_d} + \rho(x_d + t) =: \Phi_{t,x}(\xi) + \rho(x_d + t).
\end{align*}
Here, we use the notation $x = (\bar{x},x_d)$.  If $c'$ is chosen sufficiently small, then for $(t,x) \in B(0,c'R)$ and $\xi$ in the support of $\widehat{f_j}$ (in particular, $|\bar{\xi}| \leq \frac{1}{R}, |\xi_d - \rho| \leq \frac{1}{R}$) we get
\[
|\Phi_{t,x}(\xi)| \leq \frac{1}{10} 
\]
and thus
\[
|Uf_j(t,x)| =  \bigg|\int_{B(\rho e_d,\frac{1}{R})} e^{i\Phi_{t,x}(\xi)} \, \mathrm{d}\xi  \bigg| \gtrsim \frac{1}{R^d}
\]
as claimed. Taking $\lambda_j = \|f_j\|_{\dot{H}^s}^2$ and $g_j =\|f_j\|^{-1}_{\dot{H}^s} f_j$ for each $j = 1,\ldots,N$, and recalling that $N \sim R^d$, we easily deduce that
\begin{align*}
\bigg\| \sum_{j=1}^N \lambda_j |Ug_j|^2 \bigg\|_{L^{\frac{q}{2}}_tL^{\frac{r}{2}}_x} = \bigg\| \sum_{j=1}^N |Uf_j|^2 \bigg\|_{L^{\frac{q}{2}}_tL^{\frac{r}{2}}_x}  \gtrsim R^{\frac{2}{q} + \frac{2d}{r}-d}.
\end{align*}
If we assume \eqref{e:ONSwave} then 
\[
R^{\frac{2}{q} + \frac{2d}{r}-d} \lesssim \bigg(\sum_{j=1}^N \|f_j\|_{\dot{H}^s}^{2\beta}  \bigg)^{\frac{1}{\beta}} \lesssim R^{\frac{d}{\beta} - d}
\]
since $ \|f_j\|_{\dot{H}^s}^2 \sim R^{-d}$. Therefore $\frac{2}{q} + \frac{2d}{r} \leq \frac{d}{\beta}$ and this yields the necessary condition $\beta \leq \beta_{\frac{d}{2}}(q,r)$.
\begin{Acknowledgements}
The first author would like to express his appreciation to Sanghyuk Lee and Shohei Nakamura for a number of extremely helpful discussions related to this work. The second author shows his deep gratitude to Yutaka Terasawa and Satoshi Masaki for the opportunity at the RIMS conference ``Harmonic Analysis and Partial Differential Equations" in 2022.
The third author also thanks Mitsuru Sugimoto for his continuous encouragement. 
\end{Acknowledgements}


\end{document}